\documentclass[12pt,a4paper]{article}
\usepackage{amssymb}
\usepackage{amsmath}
\usepackage{amsthm}
\usepackage{indentfirst}
\usepackage{ot1patch}
\usepackage{psfrag}
\usepackage{rotate}
\usepackage{mathtools}
\usepackage{hyperref}

\newtheorem{tw}{Theorem}[section]
\newtheorem{pr}[tw]{Proposition}
\newtheorem{lm}[tw]{Lemma}

\theoremstyle{definition}

\author{\L ukasz Matysiak\\
Kazimierz Wielki University\\
Bydgoszcz, Poland \\
lukmat@ukw.edu.pl}
\title{On some properties of polynomial composites}

\begin{document}

\maketitle

\begin{abstract}
Polynomial composites were introduced by Anderson, Anderson, and Zafrullah. Over time, composites have appeared in many different papers, but they have not been sorted out in the algebra world. This paper is another part of the study of composites as an algebraic structure.
In this paper we complete possible properties for polynomial composites as ACCP, atomic, BFD, HFD, idf, FFD domains. 
In a separate section, we consider polynomial composites as Dedekind rings. 
\end{abstract}

\begin{table}[b]\footnotesize\hrule\vspace{1mm}
\end{table}

\section{Introduction}

By the ring we mean a commutative ring with unity. Let $R$ be an itegral domain.
We denote by $R^{\ast}$ the group of all invertible elements of $R$. 

\medskip

The main motivation of this paper is description polynomial composites as algebraic object. 
The related works were started in paper \cite{mm1}, where basic algebraic properties have been investigated. Continued in \cite{mm2}, where the focus was on ACCP properties and atomicity. 
This paper is the finalization of fundamental research in polynomial composites.

\medskip

D.D.~Anderson, D.F.~Anderson, M. Zafrullah in \cite{1} called object $A+XB[X]$ as a composite, where $A\subset B$ be fields. 

\medskip

There are many works where composites are used as examples to show some properties. But the most important works are presented below.

\medskip

In 1976 \cite{y1} authors considered the structures in the form $D+M$, where $D$ be a domain and $M$ be a maximal ideal of ring $R$ with $D\subset R$. 
Next, Costa, Mott and Zafrullah (\cite{y2}, 1978) considered composites in the form $D+XD_S[X]$, where $D$ be a domain and $D_S$ be a localization of $D$ relative to the multiplicative subset $S$. In 1988 \cite{y5} Anderson and Ryckaert studied classes groups $D+M$.
Zafrullah in \cite{y3} continued research on structure $D+XD_S[X]$ but he showed that if $D$ be a GCD-domain, then the behaviour of $D^{(S)}=\{a_0+\sum a_iX^i\mid a_0\in D, a_i\in D_S\}=D+XD_S[X]$ depends upon the relationship between $S$ and the prime ideals $P$ of $D$ such that $D_P$ be a valuation domain (Theorem 1, \cite{y3}).
Fontana and Kabbaj in 1990 (\cite{y4}) studied the Krull and valuative dimensions of composite $D+XD_S[X]$. 
In 1991 there was an article (\cite{1}) that collected all previous results about composites and authors began to create a further theory about composites creating results. In this paper, the considered structures were officially called composites.

\medskip

In the second section we present many properties in polynomial composites as domains. Recall, we say that an domain $R$ satysfying ACCP condition (has ACCP) if each increasing sequence of principal ideals is stationary (Proposition \ref{pr1}). 
An domain $R$ be atomic, where every nonzero noninvertible element can be presented as the product of irreducible elements (atoms) (Proposition \ref{pr1}).
The domain $R$ is a bounded factorization domain (BFD) if $R$ is atomic and for each nonzero nonunit of $R$ there is a bound on the length of factorizations into products of irreducible elements (Propositions \ref{pr2}, \ref{pr3}).  
We say that $R$ is a half-factorial domain (HFD) if is atomic and each factorization of a nonzero nonunit of $R$ into a product of irreducible elements has the same length (Propositions \ref{pr4}, \ref{pr10}).
The domain $R$ is an idf-domain (for irreducible-divisor-finite) if each nonzero element of $R$ has at most a finite number of nonassociate irreducible divisors (Propositions \ref{pr5}, \ref{pr6}). 
A domain is called finite factorization domain (FFD) if each nonzero nonunit element has only a finite number of nonassociate divisors (Proposition \ref{pr7}). In general,

$$\begin{array}{ccccccccc}
&&HFD\\
&\mbox{\rotatebox{-135}{$\Leftarrow$}} & \Uparrow & \mbox{\rotatebox{135}{$\Rightarrow$}} \\
UFD&\Rightarrow &FFD&\Rightarrow &BFD&\Rightarrow ACCP&\Rightarrow& atomic  \\
&\mbox{\rotatebox{135}{$\Leftarrow$}}&\Downarrow \\
&&idf
\end{array}$$

Recall that $R$ is an S-domain if for each height-one prime ideal $P$ of $R$, $ht P[X]=1$ in $R[X]$ (Proposition \ref{pr8}).
A commutative ring $R$ is called a Hilbert ring if every prime ideal of $R$ is an intersection of maximal ideals of $R$ (Theorem \ref{tw9}).

In Proposition \ref{pr12} we have information about composite cover.

\medskip

In the third section we have statements about polynomial composites as Dedekind domains. 
It turns out that polynomial composites of the form $K+XL[X]$ be a Dedekind rings (Theorem \ref{Dedekind}).

\section{Results}

In papers \cite{mm1} and \cite{mm2}, polynomial composites with the property of atomicity and ACCP are presented. The results below are complementary.

\begin{pr}
	\label{pr1}
	Let $T=K+XL[X]$, where $K$, $L$ are fields with $K\subset L$. Let $D$ be a subring of $K$ and $R=D+XL[X]$. Then:
	\begin{itemize}
		\item[(a) ] $R$ is atomic if and only if $T$ is atomic and $D$ is a field.
		\item[(b) ] $R$ satisfies ACCP if and only if $T$ satisfies ACCP and $D$ is a field.
	\end{itemize}
\end{pr}

\begin{proof}
	First suppose that $D$ is not a field. Then $f=d\dfrac{m}{f}$ for each $f\in XL[X]$ and $d\in D^{\ast}$. Thus no element of $XL[X]$ is irreducible ($XL[X]$ is a maximal ideal of $T$). Hence if $R$ is either atomic or satisfies ACCP, $D$ must be a field. So let $D$ be a field.
	
	\begin{itemize}
		\item[(a) ] Up to multiplication by a $\alpha\in K^{\ast}$ (resp. $\alpha\in D^{\ast}$), each element of $T$ (resp. $R$) has the form $f$ or $1+f$ for some $f\in XL[X]$. Each of these elements is irreducible in $R$ if and only if it is irreducible in $T$ (\cite{16}, Lemma 1.5; 27). If $x$ is a product of irreducibles, we may assume that each irreducible factor has the form $f$ or $1+f$ for some $f\in XL[X]$. Thus $x$ is a product of irreducible elements in $R$ if and only if it is a product of irreducible elements in $T$. Hence $R$ is atomic if and only if $T$ is atomic.
		
		\item[(b) ] We first observe that a principal ideal of $R$ or $T$ may be generated by either $f$ or $a+f$ for some $f\in XL[X]$. Let $f$, $g\in XL[X]$. It easily verified that $(1+f)R\subset (1+g)R$ if and only if $(1+f)T\subset (1+g)T$, $fR\subset (1+g)R$ if and only if $fT\subset (1+g)T$, and $fR\subset gR$ if and only if $fT\subset gT$. Also, if $fT\subset gT$, then $fR\subset (\alpha g)R$ for some $\alpha\in K^{\ast}$. Hence, to each chain of principal ideals of length $s$ in $R$ starting at $fR$ (resp., $(1+f)R$), there corresponds a chain of principal ideals of length $s$ in $T$ starting at $fT$ (resp., $(1+f)T$), and conversely. Thus $R$ satisfies ACCP if only if $T$ satisfies ACCP.
		\end{itemize} 
\end{proof}

In \cite{0} Anderson, Anderson and Zafrullah asked the following question:

\medskip

\noindent 
{\bf Question 1} If $R$ is atomic, then $R[X]$ is atomic?

\medskip

In \cite{mm2}, I considered the question and concluded that the answer was negative.

\medskip

The propositions \ref{pr2}, \ref{pr3} represent the BFD property in polynomial composites.

\begin{pr}
	\label{pr2} 
	If $A+XB[X]$ is a noetherian domain, where $A\subset B$ are domains, then $A+XB[X]$ is a BFD.
\end{pr}

\begin{proof}
	\cite{0}, Proposition 2.2.
\end{proof}

\begin{pr}
	\label{pr3} 
	Let $T=K+XL[X]$, where $K\subset L$ are fields. Let $D$ be a subring of $K$ and $R=D+XL[X]$. Then $R$ is a BFD if and only if $T$ is a BFD and $D$ is a field.
\end{pr}

\begin{proof}
	First suppose that $R$ is BFD. Then $D$ must be a field (\cite{0}, Proposition 1.2). Again from the proof of (\cite{0}, Proposition 1.2) we get that $R$ is a BFD if and only if $T$ is a BFD.  
\end{proof}

The propositions \ref{pr4} and \ref{pr10} represents the HFD property in polynomial composites.

\begin{pr}
	\label{pr4} 
	Let $T=K+XL[X]$, where $K\subset L$ are fields. Let $D$ be a subring of $K$ and $R=D+XL[X]$. Then $R$ is a HFD if and only if $D$ is a field and $T$ is a HFD. 
\end{pr}

\begin{proof}
	As in Proposition 1.2 (\cite{0}), $D$ is necessarily a field. The proof of Proposition 1.2 shows that a factorization into irreducibles in $R$ has the same length as such a factorization in $T$. Hence $R$ is a HFD if and only if $T$ is a HFD. 
\end{proof}

\begin{pr}
	\label{pr10}
	Let $A$ be a subring of a field $K$. Then $R=A+XK[X]$ is a HFD if and only if $A$ is a field.
\end{pr}

\begin{proof}
	($\Rightarrow$) Clearly, $R$ a HFD implies that $A$ is a HFD. Suppose that $A$ is not a field, so there is an irreducible element $a\in A$. Then $X=a^n(X/a^n)$ for all $n\in\mathbb{N}$. Thus $A$ must be a field.
	
	\medskip
	
	($\Leftarrow$) Suppose that $A$ is a field. By (moje) $R=A+XK[X]$ is atomic. The proof of Theorem 2.1 \cite{mm1} shows that an irreducible element of $R$ is of the for $aX$, where $a\in K$ or $a(1+Xf[X])$, where $a\in A$, $f(X)\in K[X]$, and $1+XF(X)$ is irreducible in $K[X]$. Thus for any $g(X)\in R$, the number of irreducible factors from $R$ is the same as the number of irreducible factors in a representation of $g(X)$ as a product of irreducible factors from the PID $K[X]$. Hence $R$ is a HFD.
\end{proof}

Recall that $R$ is an idf-domain if each nonzero element of $R$ has at most a finite number of nonassociate irreducible divisors. 

\begin{pr}
	\label{pr5}
	Let $T=K+XL[X]$, where $K\subset L$ are fields. Let $M$ be a subfield of $K$ and $R=M+XL[X]$. Then:
	\begin{itemize}
		\item[(a) ] Suppose that $XL[X]$ contains an irreducible element. Then $R$ is an idf-domain if and only if $T$ is an idf-domain and the multiplicative group $K^{\ast}/M^{\ast}$ is finite.
		\item[(b) ] Suppose that $XL[X]$ contains no irreducible elements. Then $R$ is an idf-domain if and only if $T$ is an idf-domain.
	\end{itemize}
\end{pr}

\begin{proof}
	(a) We first note that an element of $XL[X]$ is irreducible in $R$ if and only if it is irreducible in $T$. Let $f\in XL[X]$ be irreducible. First suppose that $R$ is an idf-domain. Then $af\mid f^2$ for all $a\in K^{\ast}$. Note that $af$ and $bf$ are irreducible in both $R$ and $T$, and that they are associates in $R$ if and only if $a$ and $b$ lie in the same coset in $K^{\ast}/M^{\ast}$. Hence $K^{\ast}/M^{\ast}$ if finite. Let $y\in T$. By multiplying by a suitable $a\in K^{\ast}$, we may assume that $y\in R$. Let $y_1, y_2, \dots, y_n$ be the distinct nonassociate irreducible divisors of $y$ in $R$. It is easily verified that any irreducible divisor of $y$ in $T$ is associated to one of the $y_i$'s. Thus $T$ is also an idf-domain. Conversely, suppose that $T$ is an idf-domain and that $K^{\ast}/M^{\ast}$ is finite. Let $z\in R$. Let $z_1, z_2, \dots, z_r$ be a complete set of nonassociate irreducible divisors of $z$ in $T$, which we may assume are all in $R$, and let $a_1, a_2, \dots, a_s$ be a set of coset representatives of $K^{\ast}/M^{\ast}$. Then any irreducible divisor of $z$ in $R$ is an associate of some $a_iz_j$. Hence $R$ is an idf-domain.
	
	\medskip
	
	(b) Since $XL[X]$ has no irreducible elements, an irreducible element in $T$ (resp., in $R$) has the form $a+f$ for some $a\in K^{\ast}$ (resp., $a\in M^{\ast}$) and $f\in XL[X]$. Hence, up to associates, each has the form $1+f$ for some $f\in XL[X]$. It is then easily verified that $\{1+f_1, 1+f_2, \dots, 1+f_n\}$ is a complete set of nonassociate irreducible divisors of a given element with respect to $R$ if and only if it is a complete set of nonassociate irreducible divisors with respect to $T$.
	
\end{proof}

\begin{pr}
	\label{pr6} 
	Let $T$ be a quasilocal integral domain of the form $K+XL[X]$, where $K\subset L$ are fields. Let $D$ be a subring of $K$ and $R=D+XL[X]$. If $D$ is not a field, then $R$ is an idf-domain if and only if $D$ has only a finite number of nonassociate irreducible elements.
\end{pr}

\begin{proof}
	Let $d$ be a nonzero nonunit of $D$. Then $f=d(f/d)$ shows that no element of $XL[X]$ is irreducible and $d$ divides each element of $XL[X]$. Also, $y=d+f=d(1+f/d)$ and $1+f/d\in R^{\ast}$ (since $T$ is quasilocal) shows that $y$ is irreducible in $R$ if and only if $d$ is irreducible in $D$. Thus $R$ is an idf-domain if and only if $D$ has only a finite number of nonassociate irreducible elements.
\end{proof}

\noindent 
{\bf Question} If $R$ is an idf-domain, then $R[X]$ be an idf-domain?

\medskip

The proposition \ref{pr7} represents the FFD property in polynomial composites.

\begin{pr}
	\label{pr7} 
	Let $T=K+XL[X]$, where $K\subset L$ are fields. Let $D$ be a subring of $K$ and $R=D+XL[X]$. Then $R$ is a FFD if and only if $T$ is a FFD, $D$ is a field, and $K^{\ast}/D^{\ast}$ is finite.
\end{pr}

\begin{proof}
	Proof is similar to \cite{0} Proposition 5.2.
\end{proof}

Recall an integral domain $D$ is called an S-domain if for each prime ideal $P$ of $D$ with $ht P=1$, $ht P[X]=1$.

\begin{lm}
	\label{l2}
	For an integral domain $D$, the following statements are equivalent.
	\begin{itemize}
		\item[(a) ] $D$ is an S-domain.
		\item[(b) ] For each prime ideal $P$ of $D$ with $ht P=1$, $D_P$ is an S-domain.
		\item[(c) ] For each prime ideal $P$ of $D$ with $ht P=1$, $\overline{D_P}$ is a Pr\"ufer domain. 
	\end{itemize}
\end{lm}

\begin{proof}
	\cite{1} Lemat 3.1
\end{proof}

\begin{lm}
	\label{l1}
	For any integral domain $D$, $D[X]$ is an S-domain.
\end{lm}

\begin{proof}
	\cite{1}, Theorem 3.2.
\end{proof}

\begin{pr}
	\label{pr8} 
	Let $D$ be an integral domain and $S$ a multiplicatively closed subset of $D$. Then $D+XD_S[X]$ is an S-domain.
\end{pr}

\begin{proof}
	Let $R=D+XD_S[X]$ and let $P$ be a height-one prime ideal of $R$. First suppose that $P\cap S\neq\emptyset$. Then $P\supseteq XD_S[X]P=XD_S[X]$. But since $ht P=1$, $P=XD_S[X]$. But then $P\cap S=\emptyset$, a cotradiction. Thus we must have $P\cap S=\emptyset$. Then $P_S$ is a height-one prime ideal in $R_S=D_S[X]$. By Lemma \ref{l1}, $R_S$ is an S-domain. Hence $R_P=R_{S_{P_S}}$ is also an S-domain by Lemma \ref{l2} (a)$\Rightarrow$(b). Thus $R$ is an S-domain by Lemma \ref{l2} (b)$\Rightarrow$(a).
\end{proof}

\medskip

Recall, a commutative ring $R$ is called a Hilbert ring if every prime ideal of $R$ is an intersection of maximal ideals of $R$. In \cite{26} it was shown that if $D\subseteq K$, where $K$ is a field, then $D+XK[X]$ is a Hilbert domain if and only if $D$ is a Hilbert domain. Thus if $D$ is a PID that is not a field and $K$ is the qoutient field of $D$, then $D+XK[X]$ is a two-dimensional, non-Noetherian, B\'ezout-Hilbert domain in which every maximal ideal is principal.

\begin{tw}
	\label{tw9} 
	Let $D$ be an integral domain and $S$ a multiplicatively closed subset od $D$ with the property that for a prime $P$ of $D$ with $P\cap S\neq\emptyset$, then $Q\cap S\neq\emptyset$ for each prime $0\neq Q\subseteq P$. Then $R=D+XD_S[X]$ is a Hilbert domain if and only if $D$ and $D_S$ are Hilbert domains.
\end{tw}

\begin{proof}
	($\Rightarrow$) Suppose that $R$ is a Hilbert domain. Then $D\cong R/XD_S[X]$ is also a Hilbert domain. Suppose that $D_S$ is not a Hilbert domain. Let $Q$ be a nonzero prime ideal od $D$ with $Q\cap S\emptyset$. Since $D$ is a Hilbet domain, $Q=\bigcap_{\alpha} M_{\alpha S}$, where $\{M_{\alpha}\}$ is the set of maximal ideals of $D$ containing $Q$. Since $Q\cap S=emptyset$ by the hypothesis on $S$, each $M_{\alpha}\cap S=\emptyset$. Hence $Q_S=\bigcap M_{\alpha S}$ is an intersection of maximal ideals of $D_S$. So every nonzero prime ideal of $D_S$ is an intersection of maximal ideals. Hence there is a nonzero element $u\in D$ such that $u$ is in every nonzero prime ideal of $D_S$. Consider $u+X\in R$. Let $P$ be prime ideal of $R$ minimal over ($u+X$) with $P\cap D=0$. (Such a prime $P$ exists since $(u+X)\cap(D\setminus\{0\})=\emptyset$). If $Q$ is a prime ideal of $R$ with $P\subsetneq Q$, then $Q\cap D\neq 0$. For otherwise in $D_S[X]$, $0\neq P_S\subsetneq Q_S$ would both contract to $0$. Now if $Q\cap S\neq\emptyset$, then $X\in XD_S[X]\subseteq Q$, while if $Q\cap S=\emptyset$, then $u\in(Q_S\cap D_S)\cap D\subseteq Q$. So every prime ideal of $R$ properly containing $P$ contains both $u$ and $X$. Hence $P$ is not the intersection of the maximal ideals containing it, contradicting the fact that $R$ is a Hilbert domain. So $D_S$ must also be a Hilbert domain.
	
	\medskip
	
	($\Leftrightarrow$) Let $Q$ be a prime ideal of $R$. Suppose that $Q\cap S\neq\emptyset$. Then $XD_S[X]=XD_S[X]Q\subseteq Q$, so $Q=Q\cap D+XD_S[X]$. Since $D$ is a Hilbert domain, $Q\cap D$ is an intersection of maximal ideals, hence so is $Q$. So we may suppose that $Q\cap S=\emptyset$. Then since $D_S[X]$ is a Hilbert domain, $Q_S=\bigcap_{\alpha} M_{\alpha}$, where $\{M_{\alpha}\}$ is the set of maximal ideals of $D_S[X]$ containing $Q_S$. Then $Q=\bigcap_{\alpha}(M_{\alpha}\cap R)$. So it suffices to show that each $M_{\alpha}\cap R$ is a maximal ideal of $R$. So let $M$ be a maximal ideal of $D_S[X]$. Then $M=N_S$, where $N$ is a prime ideal of $D[X]$. Now $M$ maximal implies $M\cap D_S$ is maximal since $D_S$ is Hilbert domain. If $M\cap D_S=0$, then $D_S$ is a field and hence $R$ is a Hilbert domain (\cite{26}, Theorem 5). So we may assume that $M\cap D_S\neq 0$. Then by hypothesis on $S$, $(M\cap D_S)\cap D=N\cap D$ must also be maximal. Since $N\supsetneq (N\cap D)[X]$, $N$ must be a maximal ideal of $D[X]$. Hence $D[X]/N\subseteq R/M\cap R\subseteq D_S[X]/M=D_S[X].N_S=D[X]/N$ since $D[X]/N$ is a field. Therefore $M\cap R$ is a maximal ideal.  
\end{proof}

The next Proposition says that every polynomial composite is a one-dimensional B\'ezout domain.

\begin{pr}
	\label{pr11}
	Let $K\subset L$ be a pair of fields with $L$ purely inseparable over $K$ (that is, $char K=p>0$ and for each $l\in L$, there exists a natural number $n=n(l)$ with $l^{p^n}\in K$). Then every ring $R$ between $K[X]$ and $L[X]$ is a one-dimensional almost B\'ezout domain.
\end{pr}

\begin{proof}
	Since $K[X]\subset L[X]$ is an integral extension, $\dim R=\dim K[X]=1$. For each $f\in L[X]$, $f^{p^n}\in K[X]$ for $n$ large enough. Hence for $f, g\in R$, $f^{p^n}, g^{p^n}\in K[X]$ for some $n\in\mathbb{N}$. But $(f^{p^n}, g^{p^n})K[X]$ is principal. Hence $(f^{p^n}, g^{p^n})R$ is principal.
\end{proof}

Let $K$ be a field, $D$ a subring of $K$.
Every ring $R$ between $D[X]$ and $K[X]$ has a composite cover, i.e. the unique minimal overring of $R$ that is a composite.
Recall $I(B,A)=\{f(X)\in B[X]\mid f(A)\subseteq A\}$.

\begin{pr}
	\label{pr12} 
	(a) Let $R$ be a domain with qoutient field $K$. Suppose that for each $0\neq r\in R$, $R/(r)$ is finite. Then the composite cover of $I(K,R)$ is $R+XK[X]$. 
	\medskip
	(b) Let $A\subseteq B$ be rings where $A$ is finite. Then the composite cover of $I(B,A)$ is $A+XB[X]$. 
\end{pr}

\begin{proof}
	(a) Let $r$ be a nonzero nonunit of $R$ and let $R/(r)=\{r_1+(r), \dots, r_n+(r)\}$. Set $f(X)=\dfrac{1}{r}(X-r_1)\dots (X-r_n)\in K[X]$. Now for $a\in R$, $a+(r)=r_i+(r)$ for some $i$, so $a-r_i=sr$ for some $s\in R$. Hence $f(a)=\dfrac{1}{r}(sr)\prod_{j\neq 1}(a-r_j)\in R$. So $f(X)=\dfrac{1}{r}X^n+\dots \in I(K,R)$ and hence $I(K,R)$ has composite cover $R+XK[X]$.
	\medskip
	(b) For each $b\in B$, $f(X)=b(\prod_{a\in A}(X-a))\in I(B,A)$.
\end{proof}

At the end of this section, we have an exact sequence.

$$0\rightarrow A+XB[X]\rightarrow B[X]\rightarrow B[X]/A+XB[X]\rightarrow 0$$

\section{Dedekind domain}

In this section we will talk about polynomial composites as Dedekind rings.

\begin{pr}
	\label{pr13} 
	Let $A\subset B$ be a pair of integral domains and let $R=A+XB[X]$.
	$R$ is integrally closed if and only if $B$ is integrally closed and $A$ is integrally closed in $B$.
\end{pr}

\begin{proof}
	\cite{1}, Theorem 2.7
\end{proof}

By Proposition \ref{pr13} if $D$ be an integral domain with qoutient field $K$ and $D\subset D_1\subset K$, then $D+XD_1[X]$ is integrally closed if and only if $D$ and $D_1$ are both integrally closed.

\begin{tw}
	\label{Dedekind}
	Let $K\subset L$ be a finite fields extension. Then $K+XL[X]$ be a Dedekind domain.
\end{tw}

\begin{proof}
	By \cite{mm1} Theorem 2.1 every nonzero prime ideal is a maximal.
	By Proposition \ref{pr13} $K+XL[X]$ is integrally closed.
	By \cite{mm3} Proposition 3.2 $K+XL[X]$ is noetherian domain.
	Hence $K+XL[X]$ be a Dedekind domain.
\end{proof}

\begin{pr}
	\label{pr14}
	Let $K\subset L$ be an extension fields and let $T=K+XL[X]$.
	\begin{itemize}
		\item[(a) ] If $P$ be a nonzero prime ideal of $T$ and $P'=\{x\in T_0; xP\subset T\}$, then $PP'=T$.
		\item[(b) ] Every nonzero ideal of $T$ has an unambiguous representation in the form product of prime ideals.
		\item[(c) ] Every nonzero ideal of $T$ is invertible.
		\item[(d) ] If $I$ is an ideal of $T$, then $T/I$ is a principal ideal domain.
		\item[(e) ] $Cl(T)$ (a group of class of invertible ideals) be isomorphic to $Pic(T)$ (a group of class of invertible modules).
		\item[(f) ] If $M$ be a finite generated torsion-free $T$-module, then $M\cong I_1\oplus I_2\oplus\dots \oplus I_k$, where $I_1$, $I_2$, $\dots$, $I_k$ are nonzero ideals of $T$ and $k$ is a rang of $M$. Moreover
		$$M\cong T^{k-1}\oplus I_1I_2\dots I_k.$$
		\item[(g) ]  If $M$ be a finite generated $T$-module, then 
		$$M\cong T^{k-1}\oplus I\oplus \bigoplus_{(P_i, n_i)} T/P_i^{n_i},$$
		where $k=\dim_{T_0}(M\otimes_T T_0)$, $I\subset T$, $I$ is unambiguously, with the accuracy to isomorphism, a designated ideal, $P_i$ are nonzero prime ideals of $T$, $n_i>0$, and a finite set of pair $(P_i, n_i)$ is designated unambiguously.
	\end{itemize}
\end{pr}

\begin{proof}
	By a Theorem \ref{Dedekind} $T=K+XL[X]$ be a Dedekind's ring.
	
	\medskip
	
	The proof of (a) -- (g) are similar to proofs in \cite{comm}, III, 3 -- 5. 
\end{proof}

The statements from \cite{mm3} are presented below. These are the characterizations of polynomial composites as Noetherian rings. It is very easy to convert a property of Noetherian into that of Dedekind. The proofs of the following is similar to Propositions in \cite{mm3}.

\begin{pr}
	\label{01}
	Let $K\subset L$ be a field extension. Put $T=K+XL[X]$. Then
	$T$ be a Dedekind domain if and only if $[L\colon K]<\infty$.
\end{pr}

\begin{pr}
	\label{02}
	Let $K\subset L$ be a fields extension such that $L^{G(L\mid K)}=K$. Put $T=K+XL[X]$. 
	$T$ be a Dedekind domain if and only if $K\subset L$ be an algebraic extension.
\end{pr}

\begin{pr}
	\label{04}
	Let $K\subset L$ be fields extension such that $K$ be a perfect field and assume that any $K$-isomorphism $\varphi\colon M\to M$, where $\varphi(L)=L$ holds for every field $M$ such that $L\subset M$.
	Put $T=K+XL[X]$. 
	$T$ be a Dedekind domain if and only if $K\subset L$ be a separable extension.
\end{pr}

\begin{pr}
	\label{06}
	Let $K\subset L$ be fields extension. Assume that if a map $\varphi\colon L\to a(K)$ is $K$-embedding, then $\varphi (L)=L$. 
	Put $T=K+XL[X]$. 
	$T$ be a Dedekind domain if and only if $K\subset L$ be a normal extension.
\end{pr}

\begin{pr}
	\label{07}
	Let $K\subset L$ be fields extension such that $L^{G(L\mid K)}=K$. Put $T=K+XL[X]$. 	
	$T$ be a Dedekind domain if and only if $K\subset L$ be a normal extension.
\end{pr}

\begin{pr}
	\label{09}
	Let $T=K+XL[X]$ be Noetherian, where $K\subset L$ be fields. Assume
	$|G(L\mid K)|=[L\colon K]$ and any $K$-isomorphism $\varphi\colon M\to M$, where $\varphi(L)=L$ holds for every field $M$ such that $L\subset M$.
	$T$ be a Dedekind domain if and only if $K\subset L$ be a Galois extension. 
\end{pr}

\begin{pr}
	\label{10}
	Let $T=K+XL[X]$, where $K\subset L$ be fields such that $K=L^{G(L\mid K)}$. $T$ be a Dedekind domain if and only if $K\subset L$ be a Galois extension. 
\end{pr}

\end{document}